\documentclass[a4 paper,12pt,bezier]{article}
\usepackage[top=3cm,bottom=3cm,left=2.5cm,right=2.5cm]{geometry}
\usepackage{amsmath}
\usepackage{amsfonts,amsthm,amssymb}
\usepackage{amsfonts}
\usepackage{graphics,subfigure,caption2}
\usepackage{tikz, color}
\usepackage{graphics,epstopdf}
\usepackage{latexsym,bm}
\usepackage{amsfonts,amsthm,amssymb,mathrsfs,bbding}
\usepackage{hyperref}
\usepackage{CJK}
\usepackage{color}
\usepackage{indentfirst}
\usepackage{graphicx}
\usepackage{cleveref}
\usepackage{booktabs}
\usepackage{multirow}
\usepackage{cleveref}

\usepackage{color}
\usepackage{enumerate}

\newcommand\red[1] {{\color{red} #1}}

\newtheorem{lem}{Lemma}[section]
\newtheorem{thm}[lem]{Theorem}
\newtheorem{cor}[lem]{Corollary}
\newtheorem{pro}[lem]{Proposition}

\newtheorem{con}[lem]{Conjecture}
\newtheorem{cla}{Claim}

\newtheorem{pb}{Problem}
\theoremstyle{plain}

\usepackage{cite}

\def \proofend {\hfill $\blacksquare$}

\begin{document}

\begin{CJK}{GBK}{song}
\title{Sufficient conditions for $(K_2 \cup kK_1)$-free graphs to be Hamilton-connected}
\author{Xiaoqiong Xu\textsuperscript{1}, Shujie Chen\textsuperscript{1}, Fengming Dong\textsuperscript{2}, Tao Tian\textsuperscript{1,}\footnote{Corresponding author. E-mails: xiaoqiongxu0512@163.com (X. Xu), shujiechen2024@163.com (S. Chen), fengming.dong@nie.edu.sg and donggraph@163.com (F. Dong), taotian0118@163.com (T. Tian).}\\
	\small  \textsuperscript{1}School of Mathematics and Statistics, Key Laboratory of Analytical Mathematics and\\
	\small Applications (Ministry of Education), Fujian Key Laboratory of Analytical Mathematics and\\
	\small Applications (FJKLAMA), Center for Applied Mathematics of Fujian Province (FJNU),\\
	\small
	Fujian Normal University,
	\small Fuzhou 350117, PR China.\\
	\small
	\textsuperscript{2}National Institute of Education, Nanyang Technological University, Singapore
}
\date{}

\maketitle {\flushleft\bf Abstract}:
The toughness of a non-complete graph $G$, denoted $\tau(G)$, is defined as
\[
\tau(G) = \min\left\{ \frac{|S|}{\omega(G-S)} : S \subseteq V(G),\ \omega(G-S) \geq 2 \right\},
\]
where $\omega(G-S)$ is the number of components of $G - S$.  For a complete graph \( G \), we define \( \tau(G) = \infty \). A graph $G$ is $t$-tough if $\tau(G) \geq t$. For a positive integer $k$, a graph $G$ is $(K_2 \cup kK_1)$-free if it contains no induced subgraph isomorphic to $K_2 \cup kK_1$. Recently, Liu \cite{liu} showed that every $2k$-connected $(K_2 \cup kK_1)$-free graph $G$ 
with $\tau(G) > 1$ is Hamilton-connected. In this paper, we strengthen this result by proving that every $(k+1)$-connected $(K_2 \cup kK_1)$-free graph 
$G$
with $\tau(G) > 1$ and minimum degree $\delta(G) \geq 2k$ is Hamilton-connected.
Moreover, by imposing restrictions to the independence number $\alpha(G)$, we prove that every \(k\)-connected \((K_2 \cup kK_1)\)-free graph $G$ of order $n$ with $2k+1 \leq \alpha(G) < \frac{n}{2}$ and  
\( \delta(G) \geq 2k\) is Hamilton-connected, and that the bounds on $\alpha(G)$ are sharp.

\maketitle {\flushleft\textit{\bf Keywords}}: Toughness, Hamilton-connected, $(K_2 \cup kK_1)$-free graph, Independence number, Minimum degree


\section{Introduction  }\label{sec1}
In this paper, we consider \emph{simple graphs}, which have neither loops nor multiple edges. 
For any graph $G$,
let  $V(G)$ and $E(G)$ be its vertex set and edge set, respectively.
Specifically, a graph \( G \) is \emph{trivial} if \( |V(G)| = 1 \).
For a vertex \( v \in V(G) \), let \( N_{G}(v) \) denote the set of neighbors of \( v \) in \( G \) and let \( d_{G}(v) = |N_{G}(v)| \) denote the \emph{degree} of \( v \) in $G$, respectively. 
For a subset  \( S \subseteq V(G) \), we define \( G\bigl[S\bigr] \) the subgraph of \( G \) induced by \( S \) and denote \( G - S = G\bigl[V(G) \setminus S\bigr] \). 
Define \( N_{G}(S) = \bigcup_{v \in S} N_{G}(v) \setminus S \) the set of neighbors of the vertex set \( S \) in \( G \).  
For a subgraph \(H\) of \(G\), we often use \(H\) to denote its vertex set \(V(H)\) when no confusion arises.

Let \( \delta(G) \) denote the \emph{minimum degree} of \( G \). The \emph{independence number} and the \emph{number of components} of \( G \) are denoted by \( \alpha(G) \) and \( \omega(G) \), respectively. We denote by \( P_n \) and \( K_n \) the \emph{path} and the \emph{complete graph} on \( n \) vertices, respectively. Let $\overline{G}$ denote the \emph{complement} of $G$.
A graph \( G \) is \emph{\(R\)-free} if it contains no induced subgraph isomorphic to \( R \).
For two vertex-disjoint graphs \( G_1 \) and \( G_2 \), their \emph{union} is denoted by \( G_1 \cup G_2 \), and their \emph{join} is denoted by \( G_1 \vee G_2 \). For a positive integer \( k \), \( kG \) denotes the disjoint union of \( k \) copies of \( G \). 
The symbol $[k]$ denotes the set $\{1, 2, \dots, k\}$.

A graph $G$ is called \emph{$k$-connected} if the removal of any set of fewer than $k$ vertices does not disconnect $G$.
A \emph{Hamilton path} (resp.\ \emph{Hamilton cycle}) in a graph \( G \) is a path (resp.\ cycle) that contains every vertex of \( G \). The graph \( G \) is \emph{Hamiltonian} if it contains a Hamilton cycle. Furthermore, a graph \( G \) is \emph{Hamilton-connected} if any two distinct vertices are connected by a Hamilton path.
For other standard graph theory notation and terminology not defined in this paper, 
we refer to \cite{Bd}.

Chv\'{a}tal~\cite{chvatal} initiated the study of Hamilton cycles via the concept of graph toughness in 1973.
The \emph{toughness} of a graph \( G \), denoted \( \tau(G) \), is defined as
\begin{equation}\label{eq01}
\tau(G) = \min \left\{ \frac{|S|}{\omega(G-S)} : S \subseteq V(G),\ \omega(G-S) \geq 2 \right\}
\end{equation} 
if \( G \) is not complete; otherwise, \( \tau(G) = \infty \). For a positive real number \( t \), the graph \( G \) is \emph{\(t\)-tough} if \( \tau(G) \geq t \). 
Furthermore, it is straightforward to show that a non-complete \( t \)-tough graph \( G \) is \( \lceil 2t \rceil \)-connected.
While every Hamiltonian graph is necessarily 1-tough, the converse fails in general. 
Chv\'{a}tal~\cite{chvatal} thus proposed the following conjecture, known as the toughness conjecture.
\begin{con}[Chv\'{a}tal~\cite{chvatal}]
	\label{conj:toughness}
	There exists a constant \( t_0 \) such that every \( t_0 \)-tough graph on at least three vertices is Hamiltonian.
\end{con}

This conjecture has spurred extensive research and has been verified for many graph classes. 
Bauer, Broersma and Veldman~\cite{9/4} 
demonstrated that 
any constant \( t_0 \) satisfying the conjecture must be at least \( \frac{9}{4} \), by constructing non-Hamiltonian \( t \)-tough graphs for all \( t < \frac{9}{4} \). 
Subsequent work has confirmed the conjecture for several families, including planar graphs \cite{planar}, claw-free graphs \cite{claw}, and chordal graphs \cite{chordal,chordal1}. 
In particular, significant progress has been made for \( R \)-free graphs, where it has been established for cases such as \( P_4 \)-free graphs~\cite{jung,p3p1+p22p1+p4}, \( 2P_2 \)-free graphs~\cite{2k21,2k22,2k23}, $( P_3 \cup P_1 )$-free graphs~\cite{p3p1+p22p1+p4}, $( P_3 \cup P_2 )$-free graphs~\cite{p2p3}, $( P_4 \cup P_1 )$-free graphs~\cite{p4p1}, $( 2P_2 \cup P_1 )$-free graphs~\cite{2p2p1}, $( P_3 \cup 2P_1 )$-free graphs~\cite{p32p1}, and $( P_2 \cup 3P_1 )$-free graphs~\cite{p23p1}.
For a comprehensive survey of related results, we refer the reader to~\cite{survey}.

In 2022, Shi and Shan~\cite{shishan} proposed the following conjecture on the Hamiltonicity of $(K_2 \cup kK_1)$-free graphs.

\begin{con}[Shi and Shan~\cite{shishan}]
	\label{conj:shi-shan}
	Let \( k \geq 4 \) be an integer and let \( G \) be a 1-tough \( 2k \)-connected \( (K_2 \cup kK_1) \)-free graph. Then \( G \) is Hamiltonian.
\end{con}

Conjecture \ref{conj:shi-shan} was independently confirmed by Xu, Li and Zhou \cite{xuleyou} and by Ota and Sanka\cite{km} as stated below.

 \begin{thm}[Xu, Li and Zhou~\cite{xuleyou}]
	Let \(k \geq 1\) be an integer and let \(G\) be a 1-tough \((K_2 \cup kK_1)\)-free graph with \(\kappa(G) > \max\{2k - 2, 2\}\). Then \(G\) is Hamiltonian.
\end{thm}

 \begin{thm}[Ota and Sanka~\cite{km}]
 Let \(k \geq 2\) be an integer and let \(G\) be a 1-tough \(k\)-connected \((K_2 \cup kK_1)\)-free graph with \(\delta(G) \geq \frac{3(k-1)}{2}\). Then \(G\) is Hamiltonian or the Petersen graph.
\end{thm}

In 2025, Hu, Wang and Shen \cite{fan} established a Fan-type sufficient condition for the Hamiltonicity in \((K_2 \cup kK_1)\)-free graphs.
Recently, Sanka \cite{sanka} proved that Conjecture \ref{conj:shi-shan} holds even for $(k-1)$-connected graphs subject to additional order and minimum degree restrictions.

\begin{thm}[Sanka~\cite{sanka}]
Let $k \geq 4$ be an integer and let $G$ be a 1-tough $(k-1)$-connected $(K_2 \cup kK_1)$-free graph of order $n \geq k^2 + k + 1$ with $\delta(G) \geq k$. Then $G$ is Hamiltonian.
\end{thm}

Motivated by these works, this paper studies sufficient conditions for the Hamilton-connectedness of \((K_2 \cup kK_1)\)-free graphs with \(k \geq 1\). 

In 1972, a classical result of Chv\'{a}tal and Erd\H{o}s~\cite{ce} states that every \(k\)-connected \(kP_1\)-free graph is Hamilton-connected. 
It is well known that a Hamilton-connected graph has toughness strictly greater than one. 
In 1978, for $P_4$-free graphs, Jung~\cite{jung} established the equivalence between Hamilton-connectedness and toughness being greater than one.

\begin{thm}[Jung~\cite{jung}] \label{jung}
Let $G$ be a $P_4$-free graph. Then $G$ is Hamilton-connected if and only if $\tau(G) > 1$.
\end{thm}

Theorem~\ref{jung} was later extended by Zheng, Broersma and Wang \cite{zbw} to graphs forbidding an induced subgraph $P_3 \cup K_1$ or $K_2 \cup 2K_1$.
 \begin{thm}[Zheng, Broersma and Wang {\cite{zbw}}] \label{zbw}
 	Let $R$ be 
 	either 
 	$P_3 \cup K_1$ or $K_2 \cup 2K_1$.
 	Then every $R$-free graph $G$ with $\tau(G) > 1$ on at least three vertices is Hamilton-connected.
 \end{thm}
 Most recently, Liu~\cite{liu} obtained the following theorem concerning the specific forbidden subgraph $K_2 \cup kK_1$.
 
\begin{thm}[Liu {\cite{liu}}] \label{liu}
	Let $k$ be a positive integer. Then every $2k$-connected $(K_2 \cup kK_1)$-free 
	graph $G$ with $\tau(G) > 1$ is Hamilton-connected.
\end{thm}

In this paper, we demonstrate the connectivity condition in Theorem \ref{liu} can be significantly reduced by imposing a minimum degree condition. Our main results are as follows.
\begin{thm} \label{thm1}
	Let $k$ be a positive integer.
Then every  \((k+1)\)-connected 
\((K_2 \cup kK_1)\)-free graph
$G$ with $\tau(G) > 1$ 
and  
\( \delta(G) \geq 2k\) is Hamilton-connected.
\end{thm}

 Lemma~\ref{lem:equiv}
 shows that if
 $G$ is $(k+1)$-connected
 and
 $(K_2 \cup kK_1)$-free,
 then
 $\alpha(G) < \frac{n}{2}$ if and only if $\tau(G) > 1$.

Since a non-complete \(\frac{k+1}{2}\)-tough graph is $(k+1)$-connected, 
we obtain
the following corollary immediately.

\begin{cor}
	For any positive integer \( k \geq 2 \), every \(\frac{k+1}{2}\)-tough \((K_2 \cup kK_1)\)-free 
	graph $G$
	with \( \delta(G) \geq 2k\) is  
	Hamilton-connected.
\end{cor}

Furthermore, if $2k+1 \leq \alpha(G) < \frac{n}{2}$, the connectivity requirement for the graph in Theorem \ref{thm1} can be relaxed from \((k+1)\)-connected to \(k\)-connected.
\begin{thm} \label{thm2}
	Let $k$ be a positive integer.
	Then every \(k\)-connected \((K_2 \cup kK_1)\)-free graph
	$G$ of order $n$ with $2k+1 \leq \alpha(G) < \frac{n}{2}$ and  
	\( \delta(G) \geq 2k\)
	 is Hamilton-connected.
\end{thm}

\noindent\textbf{Remarks.} 
\begin{enumerate}[(a)]
	\item For $k=1, 2$, every $(K_2 \cup kK_1)$-free graph with $\tau(G) > 1$ is Hamilton-connected, without requiring any minimum degree or connectivity conditions. 
For each integer $k \geq 3$, there exists a non-Hamilton-connected \((k-1)\)-connected \((K_2 \cup kK_1)\)-free graph $G$ with $\tau(G) > 1$ and \( \delta(G) \geq 2k\). 
This can be verified by the graph $M$ shown in Fig. \ref{F2},
which has the properties presented in Proposition \ref{prop2.8}.

\item Consider the balanced complete bipartite graph $G=K_{2k, 2k}$. Obviously, $G$ is a $k$-connected $(K_2 \cup kK_1)$-free graph with $\delta(G) = 2k$ and $\alpha(G) = 2k = \frac{|V(G)|}{2}$, but it is not Hamilton-connected.
This implies that the restrictions $2k+1 \leq \alpha(G) < \frac{|V(G)|}{2}$ are sharp.

\item For $n\geq 2k+1$, 
$K_n$ is a $(K_2\cup kK_1)$-free graph with $\kappa(K_n)=n-1\geq k+1$, $\alpha(K_n)=1$ and $\delta(K_n)\geq 2k$. By Theorem \ref{thm1},  $K_n$ is Hamilton-connected. However, $K_n$ does not satisfy the conditions of Theorem \ref{thm2}.
\end{enumerate}

Let $U=\overline{K}_{k}$, $U_1=U_2=\overline{K}_{k-1}$ and $W_1=W_2=K_{2k}$. We define the graph $G^*$, which is sketched in Fig. \ref{F3}, 
with vertex set  $V(G^*)=\cup_{i=1}^{2}(V(U_i)\cup V(W_i))\cup V(U)$ and edge set 
\begin{equation}\label{eq02}
E(G^*)=\bigcup_{i=1}^{2}
\big (E(W_i)\cup 
\{uw\mid u\in U\cup U_i,w\in W_i\} 
\big ).
\end{equation}
Then $G^*$ is a $(K_2\cup kK_1)$-free graph with $\kappa(G^*)=k$, $\alpha(G^*)=3k-2$ and $\delta(G^*)=2k$. Obviously, $2k+1 \leq \alpha(G^*) < \frac{7k-2}{2} = \frac{|V(G^*)|}{2}$.
By Theorem \ref{thm2}, $G^*$ is Hamilton-connected. 
However, $G^*$ does not satisfy the connectivity condition of Theorem \ref{thm1}.

Hence, for $k\geq 3$, neither Theorem \ref{thm1} nor Theorem \ref{thm2} implies  the other.
\begin{figure}[h]
	\centering
	\includegraphics[width=0.8\linewidth,trim={5cm 6.8cm 1cm 3.25cm},clip]{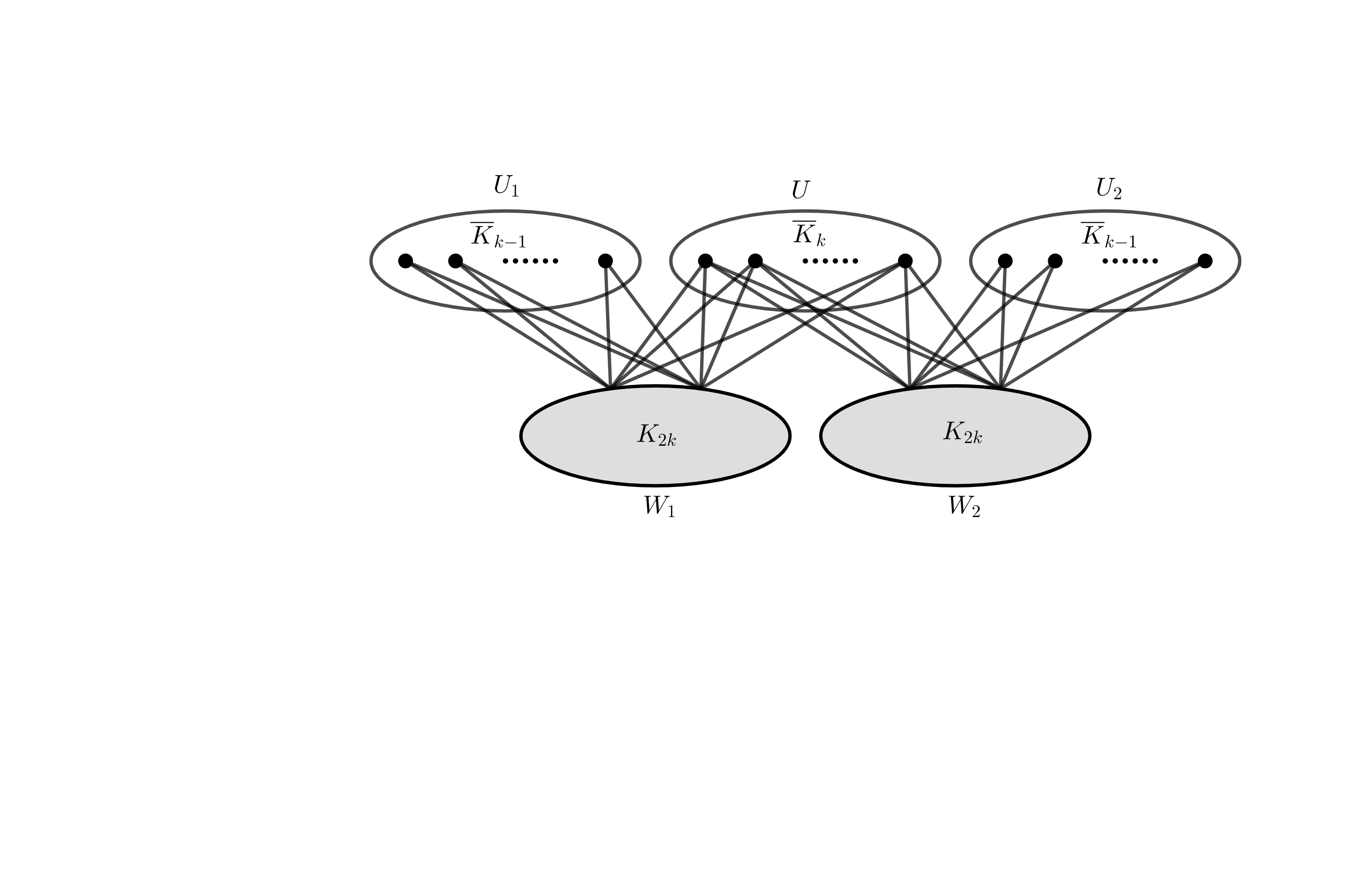}
	
	\caption{The graph $G^*$.} \label{F3}
\end{figure}

The rest of this paper is organized as follows. 
In Section \ref{sec2}, we list some lemmas that will be used in Sections \ref{sec3} and \ref{sec4} to prove our main results. Finally, in Section \ref{sec5} we conclude this paper and pose two open problems.

\section{Preliminaries  }\label{sec2}

We begin by introducing some useful notations. For two distinct vertices $x$ and $y$ in $G$, a $(x,y)$-path is a path whose endpoints are $x$ and $y$. Let $P$ be a path with a fixed orientation from its starting vertex to its end vertex. For a vertex $u \in V(P)$ and a positive integer $h$, we denote by $u^{+h}$ and $u^{-h}$ the $h$-th successor and the $h$-th predecessor of $u$ along $P$, respectively. For simplicity, we write $u^{+}$ for $u^{+1}$ and $u^{-}$ for $u^{-1}$.
For any subgraph $H$ of $G-V(P)$, 
let $N_P(H)$ be the set of vertices $x$ in $P$ such that $N_G(x)\cap V(H)\ne \emptyset$
and 
$N_P(H)^+ = \{x^+ \mid x \in N_P(H)\}$,
where \(x^+\) is the immediate successor of \(x\) on \(P\), if it exists. 
For any two vertices \( u, v \in V(P) \), the segment of \( P \) from \( u \) to \( v \) following the orientation is denoted by \( u \overrightarrow{P} v \); conversely, \( v \overleftarrow{P} u \) represents the reverse path from \( v \) to \( u \). 

The following lemmas will be used in our later proofs.
\begin{lem} [Ota and Sanka\cite{km}]\label{1.1}
	Let \( k \geq 2 \) be an integer, \( G \) be a \( (K_2 \cup kK_1) \)-free graph, and \( X \) be an independent set of \( G \). Then the following statements hold:
	\begin{enumerate}[(i)]
		\item
		for any \( v \in V(G) \), either \( N_G(v) \cap X = \emptyset \) or \( |N_G(v) \cap X| \geq |X| - k + 1 \); and
		\item for any $W\subseteq V(G)$, 
		if \( |N_G(w) \cap X| \leq |X| - k \) holds for every $w\in W$,
		then \( X \cup W \) is independent.
	\end{enumerate}
\end{lem}

\begin{lem} [Erd\H{o}s and Gallai\cite{Erdos59}]\label{1.2}
	Let $G$ be a $2$-connected graph and $x$ and $z$ be two distinct vertices of $G$. If $d_G(v)\geq d$ for every vertex $v\in V(G)\backslash\{x,z\}$, then $G$ contains an $(x,z)$-path of length at least $d$.
\end{lem}

\begin{lem}
	\label{lem:yibian}
	Let $G$ be a graph of order $n \ge 3$. If \( \tau(G) > 1 \), then \( \alpha(G) < \frac{n}{2} \).
\end{lem}
\begin{proof}
	Suppose 
	 \( \alpha(G) \geq \frac{n}{2} \).
	Let $I$ be an independent set of $G$ with  $|I|\ge \frac{n}{2}$.
	Since $n\ge 3$, $|I|\ge 2$.
	Then $S: = V(G) \setminus I$
	is a cutset of $G$. 
	Observe that $|S| = n - |I| \leq \frac{n}{2}$
	and $\omega(G - S) = |I| \geq \frac{n}{2} \geq |S|$.
	Thus, $\tau(G)\le \frac{|S|}{\omega(G - S)} \leq 1$,
	a contradiction to the given condition.
	
	This completes the proof of Lemma \ref{lem:yibian}.
\end{proof}

\begin{lem} \label{claim01}
	Let \( G \) be a $k$-connected graph, 
	where $k\ge 2$.
	Assume that $u$ and $v$ are 
	distinct vertices in $G$ 
	such that no Hamilton path of 
	$G$ connects $u$ and $v$.
	If \( P \) is a longest $(u,v)$-path in $G$ 
	and $H$ is a component of $G-V(P)$,
	then each of the following
	 holds: 
	\begin{enumerate}[(i)]
		\item	$|V(P)| \geq k$, and 
		\item $N_P(H)^+ \cap N_P(H) = \emptyset$ and $N_P(H)^+$ is an independent set.
	\end{enumerate}
\end{lem}
\begin{proof}
	$(i)$ Suppose that $|V(P)|\leq k-1$.
	Since $G$ is $k$-connected,
	$G-V(P)$ is connected and 
	 $H=G-V(P)$.
For any vertex $w$ in $P$, 
since $d_G(w)\geq k$
	and
	$|V(P)|\leq k-1$, we have 
	$|N_G(w)\cap V(H)|
	\ge d_G(w)-(|V(P)|-1)\ge 2$.

Now let \(xy\) be any edge on path \(P\).
By the conclusion in the previous paragraph, 
$|N_G(x)\cap V(H)|\ge 2$
and $|N_G(y)\cap V(H)|\ge 2$.
Thus, 
there exist distinct 
vertices $x'$ and $y'$ in $H$ 
with $x'\in N_G(x)$ and $y'\in N_G(y)$.
	Since $H$ is connected, there exists a $(x',y')$-path $Q$ in $H$.
	Then the path $P' =u \overrightarrow{P} x x'Q y' y \overrightarrow{P} v$ is a $(u,v)$-path that is longer than $P$, a contradiction.
	
	$(ii)$ Let $N_P(H) = \{x_1, \dots, x_t\}$. Since \(G\) is \(k\)-connected, by Lemma \ref{claim01}$(i)$, \(t \geq k\). Note that 
	$N_P(H)^+ \cap N_P(H) \ne  \emptyset$
	if and only if  $x_i$ and $x_i^+$ are contained in $N_P(H)$ for some $x_i\in V(P)$.
	However, if 
	$x_i, x_i^+\in N_P(H)$, then 
	there exists another $(u,v)$-path 
	$P'$ obtained from $P$ by 
	replacing edge $x_ix_i^+$ 
	by a path $x_iwP_0w'x_i^+$,
	where $x_iw,w'x_i^+\in E(G)$
	and $wP_0w'$ is a $(w,w')$-path in $H$.
	Clearly, $P'$ is a $(u,v)$-path 
	which is longer than $P$,
	a contradiction to the assumption of $P$.
	It follows immediately that $N_P(H)^+ \cap N_P(H) = \emptyset$.
	
	Suppose that there exist \( x_i^+, x_j^+ \in N_P(H)^+ \) $(i \neq j)$ with \( x_i^+ x_j^+ \in E(G) \). Without loss of generality, we assume that $i < j$. Since both $x_i$ and $x_j$ have neighbors in $H$, 
there is a $(x_i,x_j)$-path $Q$ in $G[V(H)\cup \{x_i,x_j\}]$.
Thus, the path
\(
P'= u \overrightarrow{P} x_i Q x_j \overleftarrow{P} x_i^+  x_j^+ \overrightarrow{P} v
\)
is a $(u,v)$-path that is longer than $P$, a contradiction.

This completes the proof of Lemma \ref{claim01}.\hfill 
\end{proof}

\begin{lem} \label{claim23}
	Let \( G \) be a \( (K_2 \cup kK_1) \)-free graph with \( \delta(G) \geq 2k\) and let $u,v\in V(G)$.
		Assume that 
		no Hamilton path of 
		$G$ connects $u$ and $v$,
	\( P \) is a longest $(u,v)$-path in $G$ 
	and $x$  is an isolated vertex of $G-V(P)$.
Let $N_G(x)=\{x_1, x_2, \dots, x_d\}
\subseteq V(P)$ 
and \( X = \{x\} \cup N_P(x)^+ \).
Then each of the following 
statements
holds:
	\begin{enumerate}[(i)]
		\item  for $1\le i<j\le d$
		and 
	\( ab \in E(x_i^+ \overrightarrow{P} x_j) \) 
	with \( b = a^+ \),
	either \( a x_j^+ \notin E(G) \) or
	  \( b x_i^+ \notin E(G) \); 
  
		\item for each edge $ab \in E(P)$,  $|\{a, b\} \cap N_G(X)| = 1$; and 
		
		\item 
		for any $y\in V(G)\setminus 
		(V(P)\cap N_G(X))$,
		 \( |N_G(y) \cap X| \leq 1 \).
	\end{enumerate}
\end{lem}
\begin{proof}
	(i)	Suppose that 
	\( a x_j^+ \in E(G) \) and
	\( b x_i^+ \in E(G) \).
	Then there is another $(u,v)$-path 
	\(P' = u \overrightarrow{P} x_i x x_j \overleftarrow{P} b x_i^+ \overrightarrow{P} a x_j^+ \overrightarrow{P} v\)
	that is longer than \( P \), a contradiction. 
	
	(ii) Since vertices \(a\) and \(b\) cannot both be adjacent to \(x\), we may assume without loss of generality that \( bx\notin E(G) \). Since $N_G(x) \subseteq V(P)$, \( d \geq \delta(G) \geq 2k\). Then, by Lemma \ref{claim01}$(ii)$, \( X \) is an independent set. We proceed by considering two cases based on the cardinality of \( X \).
	
	\noindent \textbf{Case 1:} \( X = \{x, x_1^+, x_2^+, \ldots, x_d^+\} \).
	
	First, if \( \{a, b\} \cap N_G(X) = \emptyset \), then for any subset \( X' \subseteq X \) with \( |X'| = k \), \( G \bigl[\{a, b\} \cup X'\bigr] \cong  K_2 \cup kK_1 \), a contradiction. Therefore, \( |\{a, b\} \cap N_G(X)| \geq 1 \).
	
	Now suppose that \( \{a, b\} \subseteq N_G(X) \). By Lemma \ref{1.1},
	\begin{equation}
		\left \{ 
		\begin{aligned}
			|N_G(a) \cap X| &\geq |X| - k + 1 = d - k + 2 , \\
			|N_G(b) \cap X| &\geq |X| - k + 1 = d - k + 2 .
		\end{aligned}
	\right.
		\label{eq1}
	\end{equation}
	
	Let \( r, r' \in [d] \) denote the maximum and minimum numbers such
	that \( ax_r^+ \in E(G) \) and \( ax_{r'}^+ \in E(G) \), respectively. Similarly, let
	\( l, l' \in [d] \) denote the minimum and maximum numbers such that
	\( bx_l^+ \in E(G) \) and \( bx_{l'}^+ \in E(G) \), respectively. 
	Then, \( r' < r \); otherwise, \( |N_G(a) \cap X| \leq 2 \), contradicting to \( |N_G(a) \cap X| \geq d - k + 2 \geq 4 \). Similarly, \( l < l' \). 
	
	Next, we show that \( r \leq l \).
	Suppose that \( l < r \).
	Then, \( ab \notin E(x_l^+ \overrightarrow{P} x_r) \); otherwise, a contradiction obtained immediately from Lemma \ref{claim23}$(i)$.
	Then, \( ab \neq x_rx_r^+ \); otherwise, \(P' = u \overrightarrow{P} x_l x a \overleftarrow{P} x_{l}^+ b \overrightarrow{P} v\) is a $(u,v)$-path with length $|V(P)|$, a contradiction. 
	\begin{figure}[htbp]
		\centering
		\includegraphics[width=1.0\linewidth,trim={1.2cm 10.5cm 1.2cm 7.1cm},clip]{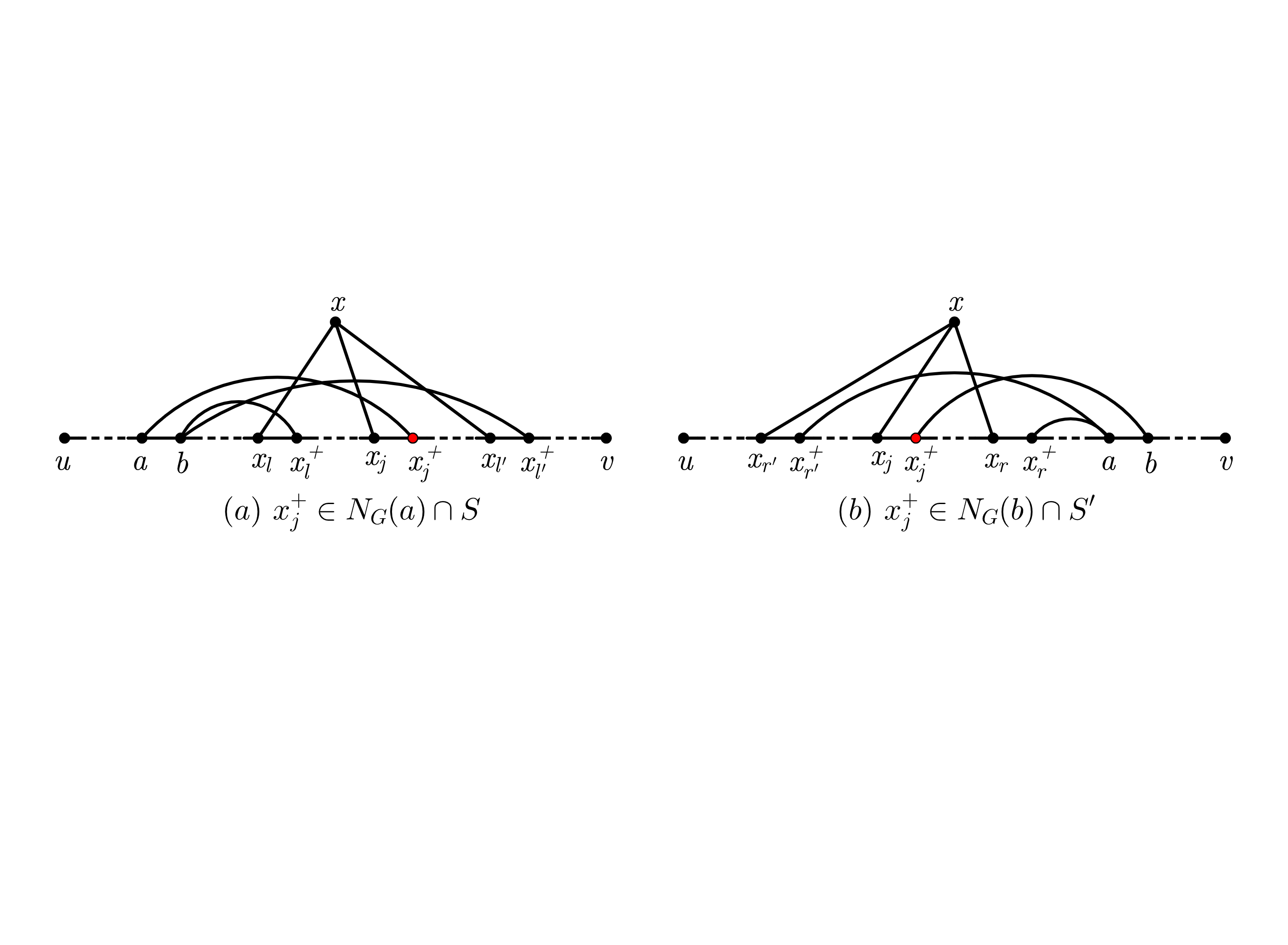}
		\caption{Constructions of a longer path.} \label{F1}
	\end{figure}
	
	When \( ab \in E(u \overrightarrow{P} x_l^+) \),  given that $|N_G(b) \cap X| \geq d - k + 2$ and $N_G(b) \cap X \subseteq \{x_j^+ :l \leq j \leq l'\}$, it follows that $l' - l \geq k+1$. Thus, the set  $S = \{x_j^+ : l \leq j < l'\}$ has size at least $k+1$. Furthermore, $N_G(a) \cap S=\emptyset$; otherwise, there exists $x_j^+ \in N_G(a) \cap S$, such that \(P' = u \overrightarrow{P} a x_{j}^+ \overrightarrow{P} x_{l'} x x_j \overleftarrow{P} b x_{l'}^+ \overrightarrow{P} v\) is a $(u,v)$-path with length $|V(P)|$, as shown in Fig. \ref{F1}$(a)$, a contradiction. Then, $l <l' \leq r$ and the edge $a x_{r}^+$ is the only possible edge in $G\bigl[S \cup \{a, x_{r}^+\}\bigr]$. This implies that $G\bigl[S \cup \{a, x_{r}^+\}\bigr]$ contains $K_2 \cup kK_1$ as an induced subgraph of $G$, a contradiction.
	
	When \( ab \in E(x_r^+ \overrightarrow{P} v) \), given that $|N_G(a) \cap X| \geq d - k + 2$ and $N_G(a) \cap X \subseteq \{x\} \cup \{x_j^+ : r' \leq j \leq r\}$, it follows that $r - r' \geq k$. Thus, the set  $S' =\{x\} \cup \{x_j^+ : r' < j \leq r\}$ has size at least $k+1$. Furthermore, $N_G(b) \cap S'=\emptyset$; otherwise, there exists $x_j^+ \in N_G(b) \cap S$, such that \(P' = u \overrightarrow{P} x_r' x x_j \overleftarrow{P} x_{r'}^+ a \overleftarrow{P} x_{j}^+ b \overrightarrow{P} v\) is a $(u,v)$-path with length $|V(P)|$, as shown in Fig. \ref{F1}$(b)$, a contradiction. Then, $ l \leq r' <r$ and the edge $b x_{l}^+$ is the only possible edge in $G\bigl[S' \cup \{b, x_{l}^+\}\bigr]$. This implies that $G\bigl[S' \cup \{b, x_{l}^+\}\bigr]$ contains $K_2 \cup kK_1$ as an induced subgraph of $G$, a contradiction.
	
	Since all cases where \( l < r \) lead to a contradiction, we conclude that \( r \leq l \). Then, \( N_G(b) \cap X \subseteq \{x_j^+ : l \leq j \leq d\} \) and \( N_G(a) \cap X \subseteq \{x\} \cup \{x_j^+ : 1 \leq j \leq r\} \). By (\ref{eq1}),
	\begin{equation}
		l \leq k - 1 \quad \text{and} \quad r \geq d - k + 1.
		\label{eq2}
	\end{equation}
	Since \( r \leq l \), \( d \leq 2k-2 \), a contradiction.
	
	\noindent \textbf{Case 2:} \( X = \{x, x_1^+, x_2^+, \ldots, x_{d-1}^+\} \).
	
	First, if \( \{a, b\} \cap N_G(X) = \emptyset \), then, $|\{a, b\} \cap N_G(X)| \geq 1$; otherwise, for any subset \( X' \subseteq X \) with \( |X'| = k \), \( G \bigl[\{a, b\} \cup X'\bigr] \cong  K_2 \cup kK_1 \), a contradiction. 
	
	Suppose first that $\{a, b\} \subseteq N_G(X)$. Then, by Lemma \ref{1.1}, 
	\begin{equation}
		\left \{ 
		\begin{aligned}
			|N_G(a) \cap X| &\geq |X| - k + 1 = d - k + 1, \\
			|N_G(b) \cap X| &\geq |X| - k + 1 = d - k + 1.
		\end{aligned}
	\right.
		\label{eq4}
	\end{equation}
	
	Let \( r, r' \in [d-1] \) denote the maximum and minimum numbers such that \( ax_r^+ \in E(G) \) and \( ax_{r'}^+ \in E(G) \), respectively. Similarly, let
	\( l, l' \in [d-1] \) denote the minimum and maximum numbers such that
	\( bx_l^+ \in E(G) \) and \( bx_{l'}^+ \in E(G) \), respectively. 
	Then, \( r' < r \); otherwise, \( |N_G(a) \cap X| \leq 2 \), contradicting the fact that \( |N_G(a) \cap X| \geq d - k + 1 \geq 3 \). 
	Similarly, \( l < l' \).
	
	The proof that \( r \leq l \) follows analogously to the argument in Case 1 by examining the position of the edge \( ab \) along the path \( P \) and deriving contradictions in each segment. Therefore, we conclude that \( r \leq l \).
	
	Then, \( N_G(b) \cap X \subseteq \{x_j^+ : l \leq j \leq d-1\} \) and \( N_G(a) \cap X \subseteq \{x\} \cup \{x_j^+ : 1 \leq j \leq r\} \). By (\ref{eq4}),
	\begin{equation}
		l \leq k - 1 \quad \text{and} \quad r \geq d - k.
		\label{eq5}
	\end{equation}
	Since \( r \leq l \), \( d \leq 2k-1\), a contradiction.
	
	Therefore, in all cases, \(|\{a,b\} \cap N_{G}(X)| = 1\) for each edge \(ab \in E(P)\).
	
	$(iii)$ If \(y \in (V(P) \setminus N_G(X)) \cup \{x\}\), then \( |N_G(y) \cap X| = 0 \). If $y \in V(G) \setminus V(P)$ and $y \neq x$, then \( |N_G(y) \cap X| \leq 1 \). Otherwise, there exist \( i, j \in [d] \) with \( i < j \), and \( y x_i^+, y x_j^+ \in E(G) \), such that \(P' = u \overrightarrow{P} x_i x x_j \overleftarrow{P} x_{i}^+ y x_{j}^+ \overrightarrow{P} v\) is a $(u,v)$-path with length $|V(P)|$, a contradiction.
	
	This completes the proof of Lemma \ref{claim23}.\hfill 
\end{proof}

\section{Proof of Theorem \ref{thm1} } \label{sec3}
\begin{lem}
	\label{lem:equiv}
	Let $G$ be a 
		\( (k+1) \)-connected \( (K_2 \cup kK_1) \)-free graph of order $n \ge 3$.
	Then \( \tau(G) > 1 \) if and only if \( \alpha(G) < \frac{n}{2} \).
\end{lem}

\begin{proof}
	The necessity follows directly from Lemma \ref{lem:yibian}. It remains to prove the sufficiency.
	Suppose that $\alpha(G) < \frac{n}{2}$. 
	For the sake of contradiction, assume that $\tau(G) \leq 1$.
	
	Since $\tau(G) \leq 1$,
	by the definition of toughness, there exists a cutset $S$ of $G$
	such that $\omega(G - S) \geq |S|$. Since $G$ is $(k + 1)$-connected, $|S| \geq k + 1$, and so $\omega(G - S) \geq k + 1$.
	Since $G$ is $(K_2 \cup kK_1)$-free, every component of $G - S$ must be trivial.
	Therefore, $I:=V(G) \setminus S$ is an independent set of $G$
		and 
		$|S|\le \omega(G - S)=|I|=n-|S|$. 
		Then
		\(
		n = |S| +|I|\leq 2|I|,
		\)
		which implies $|I| \geq \frac{n}{2}$. Then, $\alpha(G) \geq |I| \geq \frac{n}{2}$, a contradiction to the assumption.
	
	This completes the proof of Lemma \ref{lem:equiv}.
\end{proof}

\begin{lem}
	\label{lem:equiv-2}
	Let $G$ be a 
		\( (K_2 \cup kK_1) \)-free graph 
		with $n \ge 3$, $\delta(G)\ge 2k$ and $\alpha(G)<\frac n2$.
		Suppose that 
		$G$ is not Hamilton-connected,
		and  $P$ is a longest path
		among all paths in $G$ that are not Hamilton paths and no Hamilton path connecting its endpoints.
Then, $\delta(G-P)\ge 1$.	
\end{lem}

\begin{proof}
Let $u$ and $v$ be the two ends of $P$.
Suppose the conclusion fails 
and $H$ is a component of $G-P$ with $V(H)=\{x\}$.

Let 
\( X = \{x\} \cup N_P(x)^+ \). Then, by Lemma \ref{claim23}$(ii)$,
		$|\{a, b\} \cap N_G(X)| = 1$
		for each edge $ab$ on $P$,
		implying that 
		\( |W| \geq \lceil \frac{|V(P)|-1}{2} \rceil\),
		where 
		 \( W = V(P) \setminus N_G(X) \). 
		For every vertex \( y \in (V(G) \setminus V(P)) \cup W \) = \( V(G)\setminus 
		(V(P)\cap N_G(X))\), 
		by Lemma \ref{claim23}$(iii)$, \( |N_G(y) \cap X| \leq 1 \leq |X| - k \). Thus, by Lemma \ref{1.1}$(ii)$, \( X \cup (V(G) \setminus V(P)) \cup W \) = \( (V(G) \setminus V(P)) \cup W \) is an independent set in $G$. Therefore,
		\(
		\alpha(G) \geq |W| + |(V(G) \setminus V(P))| \geq \frac{n}{2},
		\)
		contradicting the given condition that $\alpha(G)<\frac n2$.

Hence the conclusion holds.
\end{proof}

\noindent \textbf{Proof of Theorem~\ref{thm1}.}
Observe that a $(K_2 \cup K_1)$-free graph is also $P_4$-free. Therefore, by Theorems \ref{jung} and \ref{zbw}, every $(K_2 \cup kK_1)$-free graph with \( \tau(G) > 1 \) is Hamilton-connected, where $k=1,2$. In the following, it suffices to consider the case $k \geq 3$.

By contradiction, suppose that $G$ is not Hamilton-connected.
Then, there exists at least one pair of distinct vertices between which no Hamilton path exists. Let $P$ be a longest path among all paths in $G$ that are not Hamilton paths and no Hamilton path connecting its endpoints.
Denote its endpoints by $u$ and $v$.
Let $H$ be a component of $G - P$.

We first show that $H=K_1$.
Suppose that there exists an edge \( zw \in E(H) \). Since $G$ is $(k+1)$-connected, $|N_P(H)^+| \geq |N_P(H)|-1 \geq k$.  
Due to the maximality of $P$, 
$N_P(H)^+ \cap N_P(H) = \emptyset$.
By Lemma \ref{claim01}$(ii)$, $G\bigl[N_P(H)^+ \cup \{z,w\}\bigr]$ contains $K_2 \cup kK_1$ as an induced subgraph of $G$, a contradiction. 
Hence $H=K_1$,
a contradiction to Lemma~\ref{lem:equiv-2}.
This completes the proof of Theorem \ref{thm1}.
\hfill $\blacksquare$

We end this section by providing an example showing that 
Theorem \ref{thm1} does not hold when the connectivity of
the graph is below $k$.

\begin{figure}[htbp]
	\centering
	\includegraphics[width=0.9\linewidth,trim={1cm 2.8cm 2cm 3cm},clip]{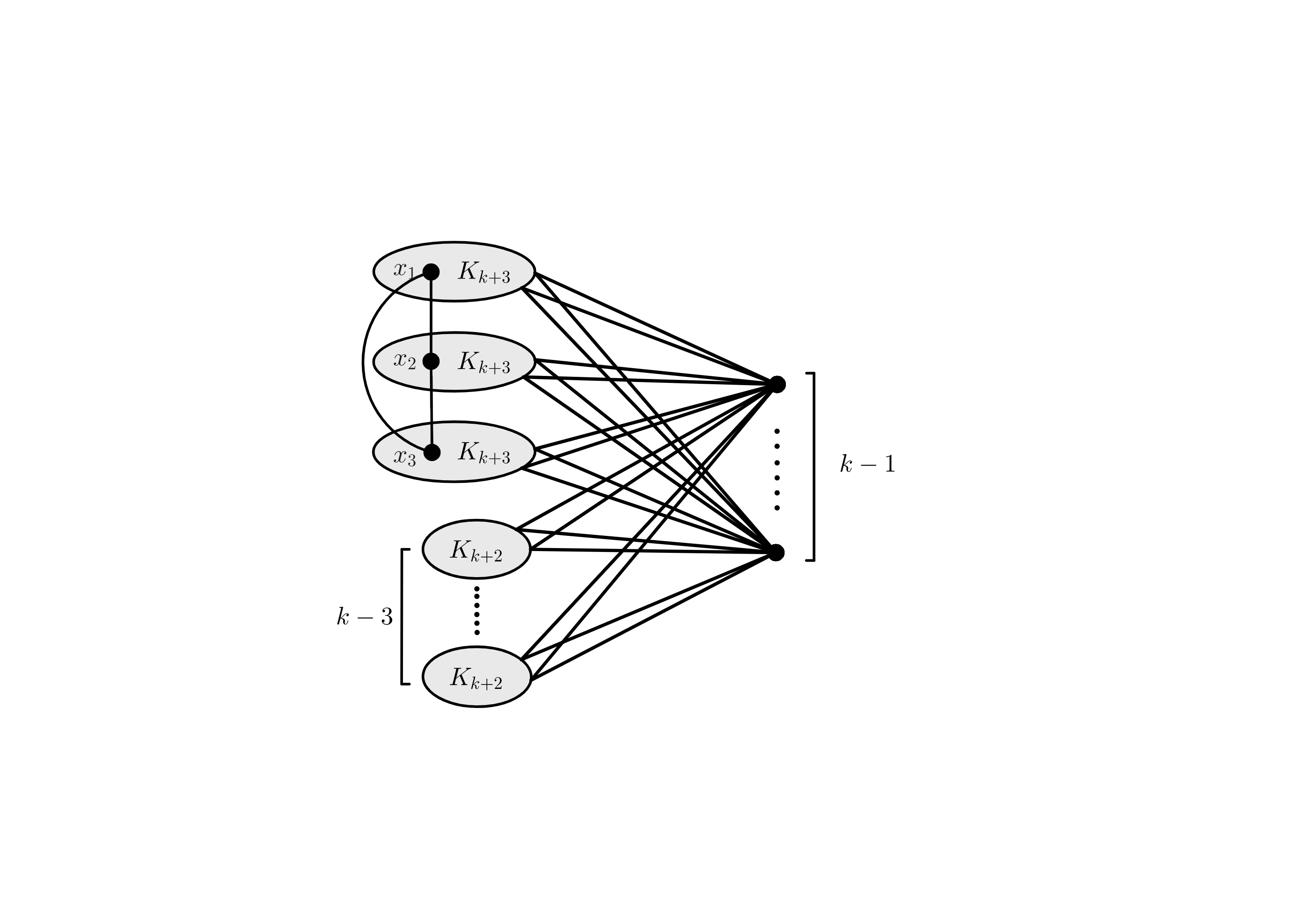}
	\caption{The graph 
		$M:=X\vee Y$, where $X$ and $Y$ are 
	defined in Page~\pageref{graphY}.} \label{F2}
\end{figure}

For each $i \in \{1,2,3\}$, let $X_i$ be a copy of $K_{k+3}$
and $x_i$ be a fixed  vertex in 
$X_i$. 
Let $X'$ be the graph obtained from $\bigcup_{i=1}^3 X_i$
by adding new edges 
$x_1x_2, x_2x_3$ and $x_3x_1$,
and let 
$X := X' \cup (k-3)K_{k+2}$. 
Let $Y$ be an empty graph 
	of order $k-1$ which is 
	disjoint from $X$.
	Now define $M$ to be the join of $X$ and $Y$, \label{graphY}
	i.e., $M:=X\vee Y$, 
	which is sketched in Fig. \ref{F2}.

\begin{pro}\label{prop2.8}
	The graph $M$ is a $(k-1)$-connected \( (K_2 \cup kK_1) \)-free graph with $\tau(M) > 1$ and $\delta(M) = 2k$.
But it is not Hamilton-connected.
\end{pro}
\begin{proof}
	Obviously, the graph \( M \) is \((k-1)\)-connected, \((K_2 \cup k K_1)\)-free, and \(\delta(M) = 2k\).
	
	Let $S\subset V(M)$ such that  $\omega(M-S)\ge 2$.
	By the structure of $M$, we know that either $V(Y)\subseteq S$ or $V(X)\subseteq S$; otherwise, $M-S$ is connected,  
	contradicting the 
		assumption that 
		$\omega(M-S)\ge 2$.

	Suppose first that $V(Y)\subseteq S$. Then $|S|\geq k-1$ and $\omega(M-S)\leq |\omega(X)|+ min\{2,t\}\leq k$, where $t=|\{x_{1},x_{2},x_{3}\}\cap S|$. In this case, $
	\frac{|S|}{\omega(M-S)} \geq \frac{|V(Y)|+t}{\omega(X)+min\{2,t\}}=
	\frac{k-1+t}{k-2+min\{2,t\}}>1.$

	Now suppose that 
	$V(X)\subseteq S$.
	Then $|S|\geq (k-3)(k+2)+3(k+3)> k(k+2)$ and $\omega(M-S)\leq |V(Y)|=k-1$. In this case, $
	\frac{|S|}{\omega(M-S)} >
	\frac{k(k+2)}{k-1}>1.$
	
	Hence, by the definition of $\tau(M)$, we have  $\tau(M)>1$.
	
	Choose two distinct vertices $u, v \in Y$.
	Suppose, for contradiction, that there exists a Hamilton path $P$ between $u$ and $v$ that traverses all vertices.
	Since \(N_M(Y) \subseteq X\) and \(Y\) is independent, the path \(P\) must alternately visit vertices from \(Y\) and \(X\).
	Note that 
		$\omega(X) = k-2 < k-1 = |Y|$.
		Thus, the subgraph of $P$ 
		induced by $X'$ must be 
		a path, 
		contradicting the fact that 
		$X'$ does not have a spanning path.
		It follows that $M$ is not Hamilton-connected.

	This completes the proof of Proposition \ref{prop2.8}.
\end{proof}

\section{Proof of Theorem \ref{thm2}} \label{sec4}

\begin{lem}	\label{k=12}
Let \(G\) be a \((K_2\cup kK_1)\)-free graph  with \(\alpha(G)\ge 2k+1\) and \(\delta(G)\ge 2k\),
	where $k\in \{1,2\}$. 
	Then \(\kappa(G)\ge k+1\).
\end{lem}
\begin{proof}
	Suppose that \(\kappa(G)\le k\). Then there exists a cutset \(S\subseteq V(G)\) with \(|S|\le k\) such that $G-S$ has at least two components. Let
	$H_1$ and $H_2$ be two distinct components of \(G-S\).
	
	For \(k=1\), we have 
		\(\delta(G)\ge 2\) and \(|S|\le 1\),
		implying that 
		each component of \(G-S\) contains at least one edge. 
		Pick \(ab\in E(H_1)\)
		and \(c\in V(H_2)\).
	Then \(G[\{a,b,c\}] \cong K_2\cup K_1\), contradiction.
	
	For \(k=2\), 
		we have \(\delta(G)\ge 4\) and \(|S|\le 2\),
		implying that $\delta(H_i)\ge 2$
		for each component $H_i$ 
		of $G-S$.
		Certainly, each $H_i$ contains at least one edge.
		Since $G$ is \((K_2\cup 2K_1)\)-free, 
		$H_1$ and $H_2$ are the only components of $G-S$
		and both of them are complete
		graphs.
	Thus, as \(|S|\le 2\), we have \(\alpha(G)\le 4\), 
	contradicting the assumption that \(\alpha(G)\ge 2k+1\ge 5\). 
	
	Therefore $\kappa(G) \ge k+1$.
\end{proof}

\noindent \textbf{Proof of Theorem~\ref{thm2}.}
By Lemma \ref{k=12} and Theorem \ref{thm1}, every $(K_2 \cup kK_1)$-free graph with $\alpha(G) < \frac{n}{2}$ is Hamilton-connected, where $k=1,2$. In the following, it suffices to consider the case $k \geq 3$.

By contradiction, suppose that $G$ is not Hamilton-connected.
Then, there exists at least one pair of distinct vertices between which no Hamilton path exists. Let $P$ be a longest path among all non-Hamilton paths in $G$ such that no Hamilton path connects its endpoints.
Denote its endpoints by $u$ and $v$.
Let $H$ be a component of $G-P$, and let $N_P(H)=\{x_1, x_2, \dots, x_d\}$.

Now we are going to establish some claims by which we complete the proof.

\begin{cla}\label{clak1}
	$H \neq K_1$, $|N_P(H)^+| = k-1$, $d=k$ and $|V(H)|\ge k+1$.
\end{cla}

\begin{proof}
By Lemma~\ref{lem:equiv-2}, 
$H\ne K_1$.	
Assume that $zw \in E(H)$.
Since $G$ is $k$-connected, by Lemma \ref{claim01}$(i)$, 
$d \geq k$ and $|N_P(H)^+| \geq d-1 \geq k-1$. 
Then $|N_P(H)^+| = k-1$;
	otherwise, by Lemma \ref{claim01}$(ii)$, $G\bigl[N_P(H)^+ \cup \{z,w\}\bigr]$ contains $K_2 \cup kK_1$ as an induced subgraph of $G$, a contradiction to the given condition.
So, $d=k$.
Since \( \delta(G) \geq 2k\), $|V(H)| \geq k+1$.
\end{proof}

\begin{cla}\label{cla0}
	$\kappa(H) \ge 2$.
\end{cla}

\begin{proof}
	Suppose that $H$ contains a cut-vertex $f$. 
	We denote the components of $H - f$ by $H_1, H_2,\ldots, H_s$, where $s \geq 2$.
	We claim that each $H_i \cong K_1~(1 \leq i \leq s)$. 
	Otherwise, without loss of generality, we assume that \( h_1h_1' \in E(H_1) \), \( h_2 \in V(H_2) \). By Lemma \ref{claim01}$(ii)$ and $|N_P(H)^+| = k-1$ by Claim~\ref{clak1}, 
	\( G\bigl[N_P(H)^+ \cup \{h_1,h_1',h_2\}\bigr] \cong K_2 \cup kK_1 \), a contradiction.
	Then each vertex of $H_i$ has degree at most $k+1$ in $G$, contradicting the 
	condition that \( \delta(G) \geq 2k\).
	Thus, Claim \ref{cla0} holds. 
\end{proof}

\begin{cla}\label{cla1}
	For any two vertices $x_i$ and $x_j$ in $N_P(H)$, where
	$1\leq i<j\leq k-1$, 
	there exist two distinct vertices $w,w'$ in $H$ such that $x_iw,x_jw'\in E(G)$.
\end{cla}
\begin{proof}
	Suppose that the vertex \(y\) is the unique common neighbor of \(x_i\) and \(x_j\) in $H$. Let 
	\(
	S = \{x_1, x_2, \dots, x_k, y\} \setminus \{x_i, x_j\}.
	\)
	Then \(G - S\) is disconnected, contradicting the condition that \(G\) is \(k\)-connected.
\end{proof}

\begin{cla}
	For any integer $i$ $(1\leq i\leq k-1)$, $|V(x_i \overrightarrow{P} x_{i+1})|\geq k+3$.
\end{cla}
\begin{proof}
	Suppose that $|V(x_i \overrightarrow{P} x_{i+1})|\leq k+2$. By Claim \ref{cla1}, there exist two distinct vertices $w,w'$ in $H$ such that $x_iw,x_{i+1}w'\in E(G)$. Since $\delta(G)\geq 2k$ and $|N_P(H)|=k$, $d_H(v)\geq k$ for all $v\in V(H)$. By Claim \ref{cla0} and Lemma \ref{1.2}, there exists a path with at least $k+1$ vertices connecting any two vertices in $H$. Let $Q$ be a path in $H$ connecting $w$ and $w'$ such that $|V(Q)| \geq k+1$. Then
	\(P' = u \overrightarrow{P} x_i wQw' x_{i+1} \overrightarrow{P} v\) is a $(u,v)$-path longer than $P$, a contradiction.
\end{proof}
\begin{cla}\label{2.8}
	For any integer $j$ $(2\leq j\leq k+1)$, $N_P(x_1^{+j}) \cap \{x_2^+,x_3^+,\dots,x_{k-1}^+\} = \emptyset$.
\end{cla}
\begin{proof}
	Suppose that \( x_1^{+j} x_i^+ \in E(G) \) for some $i \in \{2,3,\dots,k-1\}$. By Claim \ref{cla1}, there exist two distinct vertices $w,w'$ in $H$ such that $x_1w,x_iw'\in E(G)$. 
	By Claim \ref{cla0} and Lemma \ref{1.2}, there exists a path $Q$ in $H$ 
	with at least $k+1$ vertices connecting $w$ and $w'$.
	Then
	\(P' = u \overrightarrow{P} x_1 wQw' x_i \overleftarrow{P} x_1^{+j} x_i^+ \overrightarrow{P} v\) is a $(u,v)$-path longer than $P$, a contradiction.
\end{proof}
\begin{cla}
	$\alpha(H) = 1$. 
\end{cla}

\begin{proof}
	Suppose that there exist two vertices $p,q \in V(H)$ such that $pq\notin E(G)$.
	Then, by Lemma \ref{claim01}$(ii)$ and Claim \ref{2.8}, \( G\bigl[N_P(H)^+ \cup \{p,q,x_1^{+2}\}\bigr] \cong K_2 \cup kK_1 \), a contradiction.
\end{proof}
Since $\alpha(H) = 1$ and $\alpha(G) \geq 2k+1$, $\alpha(G-H) \geq 2k$ and $\alpha(G-H - N_P(H)) \geq k$.
Let $I$ be an independent set of size $k$ in $G-H - N_P(H)$. Then, by Lemma \ref{claim01}$(ii)$, \( G\bigl[I \cup \{z,w\}\bigr] \cong K_2 \cup kK_1 \), a contradiction.

This completes the proof of Theorem \ref{thm2}.
\proofend 

\section{Conclusion}
\label{sec5}
In this paper, we obtained two sufficient conditions for Hamilton-connectedness in $(K_2 \cup kK_1)$-free graphs. 
The first result (Theorem \ref{thm1}) reduces the connectivity requirement in Liu's result \cite{liu} from $\kappa(G) \ge 2k$ to $\kappa(G) \ge k+1$ by adding the minimum degree condition $\delta(G) \ge 2k$.
In Proposition~\ref{prop2.8}, 
we construct a $(k-1)$-connected graph $M$ showing that the connectivity condition in Theorem \ref{thm1} cannot be reduced to $\kappa(G) \ge k-1$. It would be natural to further consider the case of $k$-connected graphs; formally, we pose the following problem.

\begin{pb}
	For each integer $k \ge 1$, is it true that every $k$-connected $(K_2 \cup kK_1)$-free graph with $\tau(G) > 1$ Hamilton-connected?
\end{pb}
The second result (Theorem \ref{thm2}) shows that under an independence number condition $2k+1 \le \alpha(G) < \frac{n}{2}$, the connectivity can be further relaxed to $k$.
However, we are not sure whether the minimum degree condition is sharp in general, and we are
interested in whether it can be relaxed to $k$.
Finally, we pose the following problem.
\begin{pb}
For each integer $k \geq 1$, is it true that every \(k\)-connected \((K_2 \cup kK_1)\)-free graph with $2k+1 \leq \alpha(G) < \frac{n}{2}$ Hamilton-connected?
\end{pb}

\section*{ Declaration of competing interest }
There is no conflict of interest.

\section*{ Data availability }
No data was used for the research described in the paper.

\section*{ Acknowledgements }
The fourth author would like to thank the hospitality of the National Institute of
Education, Nanyang Technological University in Singapore, where part of the work was done.
This work was partly supported by the National Natural Science Foundation of China (Nos. 12101126 and 12371340), Natural Science Foundation of Fujian Province (No. 2023J01539). This work was also partly supported by China Scholarship Council (No. 202409100010).

\end{CJK}

\begin{thebibliography}{00}
\bibitem{survey}
D.~Bauer, H.J.~Broersma and E.~Schmeichel,
\newblock Toughness in graphs--a survey,
\newblock {\it Graphs Combin.}
{\bf 22} (2006) 1 -- 35.


\bibitem{9/4} D.~Bauer, H.J.~Broersma and H.J.~Veldman, Not every 2-tough graph is Hamiltonian, 
{\it Discrete Appl. Math.}
{\bf 99} (2000) 317 -- 321.

\bibitem{planar}
T.~B\"{o}hme, J.~Harant and M.~Tk\'{a}\v{c},
\newblock More than one tough chordal planar graphs are Hamiltonian,
\newblock 
{\it J. Graph Theory}
{\bf 32} (1999) 405 -- 410.



\bibitem{Bd} J.A. Bondy and U.S.R. Murty, Graph Theory, Springer, New York, 2008.


\bibitem{2k21}
H.J.~Broersma, V.~Patel and A.~Pyatkin,
\newblock On toughness and Hamiltonicity of \( 2K_2 \)-free graphs,
\newblock 
{\it J. Graph Theory}
{\bf 75} (2014) 244 -- 255.


\bibitem{chordal}
G.~Chen, M.S.~Jacobson, A.E.~K\'{e}zdy and L.~Jen\"{o},
\newblock Tough enough chordal graphs are Hamiltonian,
\newblock 
{\it Networks}
{\bf 31}  (1998) 29 -- 38.


\bibitem{chvatal} V.~Chv\'{a}tal, Tough graphs and Hamiltonian circuits, 
{\it Discrete Math.}
{\bf 5} (1973) 215 -- 228.



\bibitem{ce}
V.~Chv\'{a}tal and P.~Erd\H{o}s,
\newblock A note on Hamiltonian circuits,
\newblock {\it Discrete Math.}
{\bf 2} (1972) 111 -- 113.


\bibitem{Erdos59} P.~Erd\H{o}s and T. Gallai, On maximal paths and circuits of graphs, 
{\it Acta Math. Acad. Sci. Hungar.}
{\bf 10} (1959) 337 -- 356.



\bibitem{p32p1}
Y.~Gao and S.~Shan,
\newblock Hamiltonian cycles in 7-tough \( (P_3 \cup 2P_1) \)-free graphs,
\newblock {\it Discrete Math.}
{\bf 345} (2022) 1 -- 7.



\bibitem{p23p1}
A. Hatfield and E. Grimm, Hamiltonicity of 3-tough $(K_2 \cup 3K_1)$-free graphs, arXiv:2106.07083.



\bibitem{fan}
Z. Hu, J. Wang and C. Shen,
\newblock A Fan-type condition for cycles in 1-tough and $k$-connected $(P_{2}\cup kP_{1})$-free graphs,
\newblock 
{\it Appl. Math. Comput.}
{\bf 494} (2025) 129300.


\bibitem{claw}
H. Hui, X. Hu and W. Yang,
\newblock On the minimum degree of minimally 1-tough, triangle-free graphs and minimally $3/2$-tough, claw-free graphs,
\newblock {\it Discrete Math.}
{\bf 346} (2023) 113352.



\bibitem{jung}
H.A.~Jung,
\newblock On a class of posets and the corresponding comparability graphs,
\newblock 
{\it J. Combin. Theory Ser. B}
{\bf 24} (1978) 125 -- 133.



\bibitem{chordal1}
A.~Kabela and T.~Kaiser,
\newblock 10-tough chordal graphs are Hamiltonian,
\newblock {\it J. Combin. Theory Ser. B}
{\bf 122}  (2017) 417 -- 427.


\bibitem{p3p1+p22p1+p4}
B.~Li, H.J.~Broersma and S.~Zhang,
\newblock Forbidden subgraphs for hamiltonicity of 1-tough graphs,
\newblock 
{\it Discuss. Math. Graph Theory}
{\bf 36} (2016) 915 -- 929.

\bibitem{liu}
F. Liu,
\newblock Every \(2k\)-connected \((P_2 \cup kP_1)\)-free graph with toughness greater than one is Hamiltonian-connected,
\newblock 
{\it Bull. Korean Math. Soc.}
{\bf 63} (2026)  881 -- 889


\bibitem{2k23}
K.~Ota and M.~Sanka,
\newblock Hamiltonian cycles in 2-tough \( 2K_2 \)-free graphs,
\newblock 
{\it J. Graph Theory}
{\bf 101} (2022) 769 -- 781.


\bibitem{km}
K.~Ota and M.~Sanka,
\newblock Some conditions for Hamiltonian cycles in 1-tough \( (K_2 \cup kK_1) \)-free graphs,
\newblock {\it Discrete Math.}
{\bf 347} (2024) 113841.

\bibitem{sanka}
M.~Sanka,
\newblock On Hamiltonian cycles of 1-tough $(P_2 \cup kP_1)$-free graphs, 
\newblock arXiv:2605.19508.


\bibitem{2k22}
S.~Shan,
\newblock Hamiltonian cycles in 3-tough \( 2K_2 \)-free graphs,
\newblock {\it J. Graph Theory} {\bf 94} (2020) 349 -- 363.


\bibitem{p2p3}
S.~Shan,
\newblock Hamiltonian cycles in tough \( (P_2 \cup P_3) \)-free graphs,
\newblock 
{\it Electron. J. Combin.}
{\bf 28} (2021) \#P1.36.


\bibitem{p4p1}
S. Shan,
\newblock Hamiltonian cycles in tough $(P_{4} \cup P_{1})$-free graphs,
\newblock arXiv:2504.08936.
\bibitem{2p2p1}
S. Shan and A. Tanyel, Hamilton cycles in tough $(2P_{2} \cup P_{1})$-free graphs, arXiv:2506.12684.

\bibitem{shishan}
L.~Shi and S.~Shan,
\newblock A note on Hamiltonian cycles in 4-tough \( (P_2 \cup kP_1) \)-free graphs,
\newblock {\it Discrete Math.} {\bf 345} (2022) 113081.


\bibitem{xuleyou}
L.~Xu, C.~Li and B.~Zhou,
\newblock Hamiltonicity of 1-tough \( (P_2 \cup kP_1) \)-free graphs,
\newblock {\it Discrete Math.} {\bf 347} (2024) 113755.

\bibitem{zbw}
W.~Zheng, H.~Broersma and L.~Wang,
\newblock Toughness, forbidden subgraphs, and Hamilton-connected graphs,
\newblock 
{\it Discuss. Math. Graph Theory} {\bf 42} (2022) 187 -- 196.




	
	\end{thebibliography}
\end{document}